\documentclass{amsart}
\usepackage[utf8x]{inputenc}
\usepackage{amsmath, amsthm}

 \usepackage[usenames,dvipsnames]{pstricks}
 \usepackage{epsfig}
 \usepackage{pst-grad} 
 \usepackage{pst-plot} 

\usepackage{youngtab}

\newtheorem{theorem}{Theorem}
\newtheorem{lemma}[theorem]{Lemma}
\newtheorem{proposition}[theorem]{Proposition}
\newtheorem{corollary}[theorem]{Corollary}

\theoremstyle{definition}

\newtheorem{remark}[theorem]{Remark}

\usepackage{amscd,amssymb}

\begin{document}

\title[Cocharacters of upper triangular matrices]
{Cocharacters of polynomial identities\\
of upper triangular matrices}

\author[Silvia Boumova, Vesselin Drensky]
{Silvia Boumova, Vesselin Drensky}
\address{Institute of Mathematics and Informatics,
Bulgarian Academy of Sciences,
1113 Sofia, Bulgaria}
\email{silvi@math.bas.bg, drensky@math.bas.bg}


\maketitle

\begin{center}{To Georgi Genov -- teacher, colleague and friend,\\
on the occasion of his retirement}\end{center}

\begin{abstract}
Let $T(U_k)$ be the T-ideal of the polynomial identities of the algebra of $k\times k$ upper triangular matrices
over a field of characteristic zero. We give an easy algorithm which calculates the generating function
of the cocharacter sequence $\chi_n(U_k)=\sum_{\lambda\vdash n}m_{\lambda}(U_k)\chi_{\lambda}$ of the T-ideal $T(U_k)$.
Applying this algorithm we have found the explicit form of the multiplicities $m_{\lambda}(U_k)$ in
two cases: (i) for the ``largest'' partitions $\lambda=(\lambda_1,\ldots,\lambda_n)$
which satisfy $\lambda_{k+1}+\cdots+\lambda_n=k-1$; (ii) for the first several $k$ and any $\lambda$.
\end{abstract}

\section*{Introduction}

We fix a field $K$ of characteristic 0 and consider unital associative algebras over $K$ only.
For a background on PI-algebras and details of the results discussed
in the introduction we refer to the book \cite{D6}.
Let $R$ be a PI-algebra and  let
\[
T(R)\subset K\langle X\rangle=K\langle x_1,x_2,\ldots\rangle
\]
be the T-ideal of its polynomial identities, where $K\langle X\rangle$ is the free associative algebra
of countable rank. One of the most important objects in the quantitative study of the polynomial identities of $R$
is the cocharacter sequence
\[
\chi_n(R)=\sum_{\lambda\vdash n}m_{\lambda}(R)\chi_{\lambda},\quad n=0,1,2,\ldots,
\]
where the summation runs on all partitions $\lambda=(\lambda_1,\ldots,\lambda_n)$ of $n$
and $\chi_{\lambda}$ is the corresponding irreducible character of the symmetric group $S_n$.
The explicit form of the multiplicities $m_{\lambda}(R)$ is known for few algebras only, among them
the Grassmann algebra $E$ (Krakowski and Regev \cite{KR}, Olsson and Regev \cite{OR}),
the $2\times 2$ matrix algebra $M_2(K)$ (Formanek \cite{F1} and Drensky \cite{D2}),
the algebra $U_2(K)$ of the $2\times 2$ upper triangular matrices
(Mishchenko, Regev and Zaicev \cite{MRZ}, based on the approach of Berele and Regev \cite{BR1}, see also \cite{D6}),
the tensor square $E\otimes E$ of the Grassmann algebra (Popov \cite{P2}, Carini and Di Vincenzo \cite{CDV}),
the algebra $U_2(E)$ of $2\times 2$ upper triangular matrices
with Grassmann entries (Centrone \cite{Ce}).

The $n$-th cocharacter $\chi_n(R)$ is equal to the character of the representation of $S_n$ acting on the vector subspace
$P_n\subset K\langle X\rangle$ of the multilinear polynomials of degree $n$ modulo the polynomial identities of $R$.
It is related with another important group action, namely the action of the general linear group $GL_d=GL_d(K)$ on
the $d$-generated free subalgebra $K\langle X_d\rangle=K\langle x_1,\ldots,x_d\rangle\subset K\langle X\rangle$ modulo the polynomial
identities in $d$ variables of $R$. The algebra
\[
F_d(R)=K\langle X_d\rangle/(K\langle X_d\rangle\cap T(R))
=K\langle x_1,\ldots,x_d\rangle/(K\langle x_1,\ldots,x_d\rangle\cap T(R))
\]
is called the relatively free algebra of rank $d$ in the variety of algebras $\text{var}(R)$ generated by the algebra $R$.
It is ${\mathbb Z}_d$-graded with grading defined by
\[
\deg(x_1)=(1,0,\ldots,0),\deg(x_2)=(0,1,\ldots,0),\ldots,\deg(x_d)=(0,0,\ldots,1).
\]
The Hilbert series
\[
H(F_d(R),T_d)=H(F_d(R),t_1,\ldots,t_d)=\sum_{n_i\geq 0}\dim(F_d^{(n_1,\ldots,n_d)}(R))t_1^{n_1}\cdots t_d^{n_d}
\]
of $F_d(R)$, where $F_d^{(n_1,\ldots,n_d)}(R)$ is the homogeneous component of degree $(n_1,\ldots,n_d)$ of $F_d(R)$,
is a symmetric function which
plays the role of the character of the corresponding $GL_d$-representation. The Schur functions
$S_{\lambda}(T_d)=S_{\lambda}(t_1,\ldots,t_d)$ are the characters of the irreducible $GL_d$-submodules
$W_d(\lambda)$ of $F_d(R)$ and
\[
H(F_d(R),T_d)=\sum_{\lambda}m_{\lambda}(R)S_{\lambda}(T),\quad \lambda=(\lambda_1,\ldots,\lambda_d).
\]
By a result of Berele \cite{B1} and Drensky \cite{D1, D2},
the multiplicities $m_{\lambda}(R)$ are the same as in the cocharacter sequence $\chi_n(R)$, $n=0,1,2,\ldots$.
Hence, in principle, if we know the Hilbert series $H(F_d(R),T_d)$, we can find the multiplicities $m_{\lambda}(R)$
in $\chi_n(R)$ for those $\lambda$ which are partitions in not more than $d$ parts.
When $R$ is a finite dimensional algebra, the multiplicities $m_{\lambda}(R)$ are equal to zero for partitions
$\lambda=(\lambda_1,\ldots,\lambda_d)$, $\lambda_d\not=0$, for $d>\dim(R)$, see Regev \cite{R2}.
Hence all $m_{\lambda}(R)$ can be recovered
from $H(F_d(R),T_d)$ for $d$ sufficiently large.
Following the idea of Drensky and Genov \cite{DG1} we consider the multiplicity series of $R$
\[
M(R;T_d)=M(R;t_1,\ldots,t_d)= \sum_{\lambda}m_{\lambda}(R)T_d^{\lambda}
=\sum_{\lambda}m_{\lambda}(R)t_1^{\lambda_1}\cdots t_d^{\lambda_d}.
\]
This is the generating function of the cocharacter sequence of $R$ which corresponds to the multiplicities $m_{\lambda}(R)$
when $\lambda$ is a partition in $\leq d$ parts. Then, if we know the Hilbert series $H(F_d(R),T_d)$, the problem
is to compute the multiplicity series $M(R;T_d)$ and to find its coefficients. This problem was solved in \cite{DG2}
for rational symmetric functions of special kind and in two variables. Berele \cite{B2} suggested another approach involving
the so called nice rational functions which allowed, see Berele and Regev \cite{BR2} and Berele \cite{B3},
to solve for unital algebras the conjecture of Regev about the precise asymptotics of the growth of the codimension sequence
of PI-algebras. But the results of \cite{B2, BR2, B3} do not give explicit algorithms to find the multiplicities of
the irreducible characters. One can apply classical algorithms to find the multiplicity series when the Hilbert series
of the relatively free algebra is known. These algorithms follow from the method of Elliott \cite{E},
improved by MacMahon \cite{MM} in his ``$\Omega$-Calculus'' or Partition Analysis, with further improvements and
computer realizations, see Andrews, Paule and Riese \cite{APR} and Xin \cite{X}. See also the series of
twelve papers on MacMahon's partition analysis by Andrews, alone or jointly with Paule, Riese and Strehl (I -- \cite{A}$,\ldots,$
XII -- \cite{AP}).

Formanek \cite{F2} expressed the Hilbert series of the product of two T-ideals
in terms of the Hilbert series of the factors. Berele and Regev \cite{BR1} translated this result in the language of cocharacters.
If $\chi_n(R_1)$ and $\chi_n(R_2)$ are, respectively, the cocharacter sequences of the algebras $R_1$ and $R_2$, then
the cocharacter sequence of the T-ideal $T(R)=T(R_1)T(R_2)$ is
\[
\chi_n(R)=\chi_n(R_1)+\chi_n(R_2)+\chi_{(1)}\widehat\otimes \sum_{j=0}^{n-1}\chi_j(R_1)\widehat\otimes\chi_{n-j-1}(R_2)
-\sum_{j=0}^n\chi_j(R_1)\widehat\otimes\chi_{n-j}(R_2),
\]
where $\widehat\otimes$ denotes the ``outer'' tensor product of characters. For irreducible characters it corresponds
to the Littlewood-Richardson rule for products of Schur functions:
\[
\chi_{\lambda}\widehat\otimes\chi_{\mu}=\sum_{\nu\vdash\vert\lambda\vert+\vert\mu\vert}c_{\lambda\mu}^{\nu}\chi_{\nu},
\]
where
\[
S_{\lambda}(T_d)S_{\mu}(T_d)=\sum_{\nu\vdash\vert\lambda\vert+\vert\mu\vert}c_{\lambda\mu}^{\nu}S_{\nu}(T_d),
\]
\[
\lambda=(\lambda_1,\ldots,\lambda_p),\quad \mu=(\mu_1,\ldots,\mu_q),\quad d\geq p+q.
\]

In the present paper we study the cocharacter sequence of the algebra $U_k=U_k(K)$ of $k\times k$ upper triangular matrices.
The algebra $U_k$ is one of the central objects in the theory of PI-algebras satisfying a nonmatrix polynomial identity
(i.e., an identity which does not hold for the $2\times 2$ matrix algebra $M_2(K)$).
Latyshev \cite{L1} proved that every finitely generated PI-algebra with a nonmatrix identity
satisfies the identities of $U_k$ for a suitable $k$.
Hence the polynomial identities of $U_k$ may serve as a measure of the complexity of the polynomial identities
of finitely generated algebras with nonmatrix identity in the same way as the polynomial identities of the $k\times k$ matrix
algebra $M_k(K)$ measure the complexity of the identities of arbitrary PI-algebras, see \cite{L3}.

Yu. Maltsev \cite{Ma} showed that the polynomial identities of $U_k$ follow from the identity
\[
[x_1,x_2]\cdots [x_{2k-1},x_{2k}]=0,
\]
where $[x,y]=xy-yx$ is the commutator of $x$ and $y$. This means that $T(U_k)=C^k$, where
\[
C=T(K)=K\langle X\rangle[K\langle X\rangle,K\langle X\rangle]K\langle X\rangle
\]
is the commutator ideal of $K\langle X\rangle$.
Using purely combinatorial methods (the technique of partial ordered sets due to Higman \cite{Hi} and Cohen \cite{C})
Genov \cite{G1, G2} and Latyshev \cite{L2} proved that
every algebra satisfying the identities of $U_k$ has a finite basis of its polynomial identities.
Later this result was generalized by Latyshev \cite{L4} and Popov \cite{P1} for PI-algebras satisfying the identity
\[
[x_1,x_2,x_3]\cdots [x_{3k-2},x_{3k-1},x_{3k}]=0
\]
which generates the T-ideal $T(U_k(E))=T^k(E)$ of the algebra $U_k(E)$
of $k\times k$ upper triangular matrices with entries from the Grassmann algebra $E$. For long time, until Kemer
developed his structure theory of T-ideals and solved the Specht problem (i.e., finite basis problem) for arbitrary PI-algebras
(\cite{K1, K2}, see also his book \cite{K3} for an account of the theory), the theorems of Genov, Latyshev and Popov
\cite{G1, G2, L2, L4, P1} covered all known examples of classes of PI-algebras with the finite basis property.

Using methods of commutative algebra Krasilnikov \cite{Kr} established that every Lie algebra satisfying the Lie polynomial identities of $U_k$
has a finite basis of its Lie identities. His approach works also for associative algebras and gives a simple proof of the result
of Genov \cite{G1, G2} and Latyshev \cite{L2}. Drensky \cite{D3}, using the method of Krasilnikov \cite{Kr} showed that
the Hilbert series of every finitely generated relatively free algebra with nonmatrix polynomial identity is a rational function.
Later this fact was generalized by Belov \cite{Be} (see \cite{KBR} for detailed exposition)
for arbitrary finitely generated relatively free algebras, using the theory of Kemer \cite{K3}.

The T-ideals of the algebras $U_k(K)$ and $U_k(E)$ have another interesting property \cite{D4}.
They are examples of maximal T-ideals of a given exponent of the codimension sequences (and the corresponding
varieties of algebras are minimal varieties of this exponent). Later, Giambruno and Zaicev \cite{GZ1, GZ2}
used the theory of Kemer \cite{K3} combined with methods of representation theory of the symmetric groups and
proved the conjecture of \cite{D4} that the only maximal T-ideals of a given exponent are the products
$T(R_1)\cdots T(R_k)$, where $T(R_i)$ are T-prime T-ideals from the structure theory of T-ideals developed by Kemer \cite{K1}.

In this paper we use the result of Formanek \cite{F2}
and calculate the Hilbert series $H(F_d(U_k),T_d)$ for any $k$ and $d$:
\[
H(F_d(U_k),T_d)=\frac{1}{t_1+\cdots +t_d-1}\left(\left(1+(t_1+\cdots+t_d-1)\prod_{i=1}^d\frac{1}{1-t_i}\right)^k-1\right)
\]
\[
=\sum_{j=1}^k\binom{k}{j}\left(\prod_{i=1}^d\frac{1}{1-t_i}\right)^j(t_1+\cdots+t_d-1)^{j-1}.
\]
Hence $H(F_d(U_k),T_d)$ is a linear combination of expressions of the form
\[
\left(\prod_{i=1}^d\frac{1}{1-t_i}\right)^p(t_1+\cdots+t_d)^q,\quad 0\leq q<p\leq k.
\]
If two symmetric functions $f(T_d)$ and $g(T_d)$ are related by
\[
f(T_d)=g(T_d)\prod_{i=1}^d\frac{1}{1-t_i},
\]
then $f(T_d)$ is Young-derived from $g(T_d)$ and the decomposition of $f(T_d)$ as a series of Schur functions can be
obtained from the decomposition of $g(T_d)$ using the Young rule.

Using the decomposition
\[
(t_1+\cdots+t_d)^q=\sum_{\lambda\vdash q}d_{\lambda}S_{\lambda}(T_d),
\]
where $d_{\lambda}$ is the degree of the irreducible $S_q$-character $\chi_{\lambda}$,
it is sufficient to apply the Young rule up to $k$ times on the Schur functions $S_{\lambda}(T_d)$
for all partitions $\lambda$ of $q\leq k-1$.
It has turned out that if $m_{\lambda}(U_k(K))$ is different from zero, then $\lambda=(\lambda_1,\ldots,\lambda_{2k-1})$
is a partition in not more than
$2k-1$ parts and $\overline{\lambda}=(\lambda_{k+1},\ldots,\lambda_{2k-1})$ is a partition of $i\leq k-1$.

We have found the exact value of $m_{\lambda}(U_k(K))$ for the ``largest'' partitions with $\overline{\lambda}$ a partition of $k-1$.
We use the fact that in the decomposition
\[
 \prod_{i=1}^k \frac{1}{(1-t_i)^k} = \sum_{\mu} n_{\mu}S_{\mu}(T_k),\quad \mu=(\mu_1,\ldots,\mu_k),
\]
the coefficient $n_{\mu}$ of $S_{\mu}(T_k)$
is equal to the dimension of the irreducible $GL_k$-module $W_k(\mu)$ corresponding to the partition $\mu$.
We obtain for $\overline{\lambda}\vdash k-1$,
\[
m_{\lambda}(U_k)=d_{\overline{\lambda}}\dim(W_k(\lambda_1,\ldots,\lambda_k)),
\]
where $d_{\overline{\lambda}}$ is the degree of the $S_{k-1}$-character $\chi_{\overline{\lambda}}$ and
\[
\dim(W_k(\lambda_1,\ldots,\lambda_k))=S_{\lambda}(\underbrace{1,\ldots,1}_{k\text{ \rm times}})
=\prod_{1\leq i<j\leq k}  \frac{\lambda_i - \lambda_j +j -i}{j-i}.
\]

There is an easy algorithm from \cite{DG1} with input
the multiplicity series of a symmetric function and output the multiplicity series of its Young-derived.
Applying it, we have found the explicit form of the multiplicity series of $H(F_d(U_k),T_d)$
and the multiplicities $m_{\lambda}(U_k)$ for the first several $k$ and any $\lambda$.

The main results of this paper have been announced without proofs in \cite{BD}.

\section{Preliminaries}

We fix a positive integer $d$ and consider the algebra
\[
{\mathbb C}[[T_d]]={\mathbb C}[[t_1,\ldots,t_d]]
\]
of formal power series in $d$ commuting variables.
As usually, if
\[
f(T_d)=\frac{p(T_d)}{q(T_d)}\in {\mathbb C}(T_d),\quad p(T_d),q(T_d)\in {\mathbb C}[T_d],q(T_d)\not=0,
\]
is a rational function and $g(T_d)\in {\mathbb C}[[T_d]]$ is such that
$p(T_d)=q(T_d)g(T_d)$ in ${\mathbb C}[[T_d]]$, we shall identify $f(T_d)$ and $g(T_d)$.
Let ${\mathbb C}[[T_d]]^{S_d}$ be the subalgebra of symmetric functions. Every symmetric function $f(T_d)\in {\mathbb C}[[T_d]]^{S_d}$
can be presented in the form
\[
f(T_d)=\sum_{\lambda}m_{\lambda}S_{\lambda}(T_d),\quad m_{\lambda}\in {\mathbb C}, \lambda=(\lambda_1,\ldots,\lambda_d),
\]
where $S_{\lambda}(T_d)$ is the Schur function related to $\lambda$. For details on the theory of Schur functions see
the book by Macdonald \cite{Mc}.
As usually, we shall omit the zeros in the partitions and shall identify $(\lambda_1,\ldots,\lambda_i,0,\ldots,0)$
and $(\lambda_1,\ldots,\lambda_i)$.
There are several ways to define Schur functions. The most convenient for our purposes is
to define them as fractions of Vandermonde type determinants:
\[
S_{\lambda}(T_d)=\frac{V(\lambda+\delta)}{V(\delta)},
\]
where $\delta=(d-1,\ldots,2,1)$ and for $\mu=(\mu_1,\ldots,\mu_d)$
\[
V(\mu,T_d)=\left\vert\begin{matrix}
t_1^{\mu_1}&t_2^{\mu_1}&\cdots&t_d^{\mu_1}\\
&&&\\
t_1^{\mu_2}&t_2^{\mu_2}&\cdots&t_d^{\mu_2}\\
&&&\\
\vdots&\vdots&\ddots&\vdots\\
&&&\\
t_1^{\mu_d}&t_2^{\mu_d}&\cdots&t_d^{\mu_d}\\
\end{matrix}\right\vert.
\]

\noindent
We shall denote by $[\lambda]$ the Ferrers (or Young) diagram of $\lambda$.
Recall that the $\lambda$-tableau $D_{\lambda}$ is semistandard if its entries do not decrease in rows
reading from left to right, and increase strictly in columns reading from top to bottom.
The tableau is of contents $\alpha=\alpha(D_{\lambda})=(\alpha_1,\ldots,\alpha_d)$ if each integer $i=1,\ldots,d$
appears in the tableau exactly $\alpha_i$ times.

\bigskip
\begin{center}
\young(1122234,22344,3667,57)\\
\bigskip
A semistandard (7,5,4,2)-tableau\\
of contents $(2,5,3,3,1,2,2)$.
\end{center}
\bigskip

\noindent
Another presentation of Schur functions is given in terms
of semistandard Young tableaux:
\[
S_{\lambda}(T_d)=\sum T^{\alpha(D_{\lambda})},
\]
where the summation runs on all semistandard $\lambda$-tableaux.

\begin{center}
\young(11,2) \, \young(11,3) \, \young(12,2) \, \young(13,3) \, \young(22,3) \, \young(23,3) \, \young(12,3) \, \young(13,2)\\
\bigskip
$\displaystyle S_{(2,1)}(T_3)=\sum_{i\not= j}t_i^2t_j+2t_1t_2t_3$.
\end{center}
\bigskip

\noindent
It is well known that Schur functions play the role of characters of irreducible polynomial representations
of $GL_d$. If $W_d(\lambda)$ is the irreducible $GL_d$-module labeled by the partition $\lambda$, then
\[
\dim(W_d(\lambda))=S_{\lambda}(\underbrace{1,\ldots,1}_{d\text{ \rm times}})
=\prod_{1\leq i<j\leq d}\frac{\lambda_i-\lambda_j+j-i}{j-i}.
\]

We associate with
\[
f(T_d)=\sum_{\lambda}m_{\lambda}S_{\lambda}(T_d)
\]
its multiplicity series
\[
M(f;T_d)=\sum_{\lambda}m_{\lambda}T_d^{\lambda}
=\sum_{\lambda}m_{\lambda}t_1^{\lambda_1}\cdots t_d^{\lambda_d}\in {\mathbb C}[[T_d]].
\]
It is also convenient to consider the subalgebra ${\mathbb C}[[V_d]]\subset {\mathbb C}[[T_d]]$
of the formal power series in the new set of variables $V_d=\{v_1,\ldots,v_d\}$, where
\[
v_1=t_1,v_2=t_1t_2,\ldots,v_d=t_1\cdots t_d.
\]
Then the multiplicity series $M(f;T_d)$ can be written as
\[
M'(f;V_d)=\sum_{\lambda}m_{\lambda}v_1^{\lambda_1-\lambda_2}\cdots
v_{d-1}^{\lambda_{d-1}-\lambda_d}v_d^{\lambda_d}\in {\mathbb C}[[V_d]].
\]
We also call $M'(f;V_d)$ the multiplicity series of $f$. The advantage of the mapping
$M':{\mathbb C}[[T_d]]^{S_d}\to {\mathbb C}[[V_d]]$ defined by
$M':f(T_d)\to M'(f;V_d)$ is that it is a bijection.

For a PI-algebra $R$ we define the multiplicity series of $R$
\[
M(R;T_d)=M(R;t_1,\ldots,t_d)= \sum_{\lambda}m_{\lambda}(R)T_d^{\lambda}
=\sum_{\lambda}m_{\lambda}(R)t_1^{\lambda_1}\cdots t_d^{\lambda_d}.
\]
Similarly we define the series $M'(R;V_d)$.

\begin{lemma}\label{relation between M and YM}
{\rm (Berele \cite{B2})}
The functions $f(T_d)\in {\mathbb C}[[T_d]]^{S_d}$ and $M(f;T_d)$ are related by the following equality. If
\[
f(T_d)\prod_{i<j}(t_i-t_j)=\sum_{p_i\geq 0}b(p_1,\ldots,p_d)t_1^{p_1}\cdots t_d^{p_d},
\quad b(p_1,\ldots,p_d)\in {\mathbb C},
\]
then
\[
M(f;T_d)=\frac{1}{t_1^{d-1}\cdots t_{d-2}^2t_{d-1}}
\sum_{p_i>p_{i+1}}b(p_1,\ldots,p_d)t_1^{p_1}\cdots t_d^{p_d},
\]
where the summation is on all $p=(p_1,\ldots,p_d)$ such that
$p_1>p_2>\cdots>p_d$.
\end{lemma}

\begin{remark}\label{how to check the proof}
In the general case, it is difficult to find $M(f;T_d)$ if we know $f(T_d)$. But
it is very easy to check whether the formal power series
\[
h(T_d)=\sum h(q_1,\ldots,q_d)t_1^{q_1}\cdots t_d^{q_d},\quad
q_1\geq\cdots\geq q_d,
\]
is equal to the multiplicity series $M(f;T_d)$ of $f(T_d)$ because
$h(T_d)=M(f;T_d)$ if and only if
\[
f(T_d)\prod_{i<j}(t_i-t_j)=\sum_{\sigma\in S_d}\text{\rm sign}(\sigma)
t_{\sigma(1)}^{d-1}t_{\sigma(2)}^{d-2}\cdots t_{\sigma(d-1)}
h(t_{\sigma(1)},\ldots,t_{\sigma(d)}).
\]
This equation can be used to verify the computational results on multiplicities.
\end{remark}

The Young rule in representation theory of the symmetric groups describes
in the language of Young diagrams the induced $S_{m+n}$-character of the
$S_m\times S_n$-character $\chi_{(m)}\otimes\chi_{\mu}$, $\mu\vdash n$
(which is equal to the outer tensor product $\chi_{(m)}\widehat\otimes \chi_{\mu}$).
In the special case $m=1$ it is equivalent to the Branching theorem for the induced $S_{n+1}$-character
of $\chi_{\mu}$, $\mu\vdash n$.
Translated in the language of Schur functions the Young rule is stated as
\[
S_{(m)}(T_d)S_{\mu}(T_d)=\sum_{\lambda}S_{\lambda}(T_d),
\]
where the summation is over all partitions $\lambda$ such that the skew diagram
$[\lambda/\mu]$ is a horizontal strip of size $m$, i.e.,
\[
\lambda_1+\cdots+\lambda_d=\mu_1+\cdots+\mu_d+m,
\]
\[
\lambda_1\geq\mu_1\geq \lambda_2\geq\mu_2\geq\cdots\geq \lambda_d\geq\mu_d.
\]
Regev \cite{R3} introduced the notion of Young-derived sequences of $S_n$-characters.
The sequence $\zeta_n$, $n=0,1,2,\ldots$, is Young-derived if it is obtained from another sequence
of $S_k$-characters $\xi_k$, $k=0,1,2,\ldots$, by applying the Young rule:
\[
\zeta_n=\sum_{k=0}^n\xi_{(n-k)}\widehat\otimes\xi_k,\quad n=0,1,2,\ldots
\]
In terms of symmetric functions this means that the symmetric function
\[
f(T_d)=\sum_{\lambda}m_{\lambda}S_{\lambda}(T_d)
\]
is Young-derived from
\[
g(T_d)=\sum_{\lambda}p_{\mu}S_{\mu}(T_d)
\]
if and only if the multiplicities $m_{\lambda}$ and $p_{\mu}$ are related with the condition
\[
m_{(\lambda_1,\ldots,\lambda_d)}=\sum p_{(\mu_1,\ldots,\mu_d)},
\quad
\lambda_1\geq\mu_1\geq \cdots\geq \lambda_d\geq\mu_d.
\]
The well known equality
\[
\prod_{i=1}^d\frac{1}{1-t_i}=\sum_{m\geq 0}S_{(m)}(T_d)
\]
gives that this is equivalent to the equality
\[
f(T_d)=g(T_d)\prod_{i=1}^d\frac{1}{1-t_i}.
\]
Let $Y$ be the linear operator in ${\mathbb C}[[V_d]]\subset {\mathbb C}[[T_d]]$ which sends the multiplicity series of
the symmetric function $g(T_d)$ to the multiplicity series of its Young-derived
$f(T_d)$:
\[
Y(M(g);T_d)=M(f;T_d)=M\left(g(T_d)\prod_{i=1}^d\frac{1}{1-t_i};T_d\right).
\]
The following proposition describes the multiple action of $Y$ on 1. We believe that it is folklorically known
although we were not able to find references.

\begin{proposition}\label{result about multiplicities}
For $d\geq k\geq 1$ the following decomposition holds
\[
\prod_{i=1}^d\frac{1}{(1-t_i)^k}=\sum_{\mu}n_{\mu}S_{\mu}(T_d),
\]
where the summation is on all partitions $\mu=(\mu_1,\ldots,\mu_k)$ and
\[
n_{\mu}=S_{\mu}(\underbrace{1,\ldots,1}_{k\text{ \rm times}})=\dim(W_k(\mu)).
\]
Equivalently,
\[
Y^k(1)=\sum_{\mu}\dim(W_k(\mu))T_k^{\mu},
\quad \mu=(\mu_1,\ldots,\mu_k),\quad k\geq 1.
\]
\end{proposition}

\begin{proof}
Clearly
\[
\prod_{i=1}^d\frac{1}{(1-t_i)^k}=\sum S_{(m_1)}(T_d)\cdots S_{(m_k)}(T_d),
\]
where the summation is on all $k$-tuples of nonnegative integers $(m_1,\ldots,m_k)$.
By the Young rule
\[
S_{(m_1)}(T_d)S_{(m_2)}(T_d)=\sum S_{\pi}(T_d),
\]
where the sum is on all partitions $\pi=(\pi_1,\pi_2)\vdash m_1+m_2$ such that
$\pi_1\geq m_1$ and the skew diagram $[\pi/(m_1)]$ is a horizontal strip.
We fill in the entries of $[(m_1)]$ and $[(m_2)]$ with 1's and 2's, respectively.
Then we fill in with 1's and 2's the boxes of $[\pi]$ corresponding
to the boxes of $[(m_1)]$ and $[(m_2)]$, respectively:
\bigskip
\begin{center}
{\normalsize \Yvcentermath1
$\displaystyle \young(1\cdots1)\otimes \young(2\cdots 2) =
\sum \young(1\cdots11\cdots12\cdots2,2\cdots2)
$
}
\end{center}

\bigskip

\noindent As a result, we obtain a bijection between the summands $S_{\pi}(T_d)$ in the decomposition of the product
$S_{(m_1)}(T_d)S_{(m_2)}(T_d)$ and the semistandard tableaux of content $(m_1,m_2)$. In the next step, the product
of three Schur functions has the form
\[
S_{(m_1)}(T_d)S_{(m_2)}(T_d)S_{(m_3)}(T_d)=\sum S_{\rho}(T_d),
\]
where the sum is on all partitions $\rho=(\rho_1,\rho_2,\rho_3)\vdash m_1+m_2+m_3$ which contain
a partition $\pi=(\pi_1,\pi_2)\vdash m_1+m_2$ such that the skew diagrams $[\pi/(m_1)]$
and $[\rho/\pi]$ are horizontal strips. The Schur function $S_{\rho}(T_d)$ participates in the sum
as many time as the possible ways to choose the partition $\pi$. Hence $S_{\rho}(T_d)$
appears in the sum with its multiplicity in the decomposition of $S_{(m_1)}(T_d)S_{(m_2)}(T_d)S_{(m_3)}(T_d)$.
Again, filling in the entries of $[(m_1)]$, $[(m_2)]$ and $[(m_3)]$ with 1's, 2's and 3's, respectively,
we obtain a bijection between the summands $S_{\rho}(T_d)$ of $S_{(m_1)}(T_d)S_{(m_2)}(T_d)S_{(m_3)}(T_d)$
and the semistandard tableaux of content $(m_1,m_2,m_3)$:

\bigskip
\begin{center}
{\normalsize \Yvcentermath1
$\displaystyle \young(1\cdots1)\otimes \young(2\cdots 2)\otimes \young(3\cdots 3) =
\sum \young(1\cdots11\cdots11\cdots22\cdots23\cdots3,2\cdots22\cdots23\cdots3,3\cdots3)
$
}
\end{center}

\bigskip

\noindent This bijection preserves the shape of the partitions and $S_{\rho}(T_d)$ is mapped to a $\rho$-tableau.
Continuing in this way we obtain a bijection between the summands $S_{\mu}(T_d)$ in the decomposition of the product
$S_{(m_1)}(T_d)\cdots S_{(m_k)}(T_d)$ and the semistandard tableaux of content $(m_1,\ldots,m_k)$.
This bijection counts the multiplicity of $S_{\mu}(T_d)$ and preserves the shape of $\mu$.
Hence the multiplicity of $S_{\mu}(T_d)$ in the decomposition of $\prod_{i=1}^d1/(1-t_i)^k$ is equal to the number of
semistandard $\mu$-tableaux which is equal to $S_{\mu}(1,\ldots,1)$ and to the dimension of the $GL_k$-module $W_k(\mu)$.
The equivalence of both statements of the proposition is obvious.
\end{proof}

\begin{remark}\label{comments on classical invariant theory}
The multiplicities in the decomposition of $\prod 1/(1-t_i)^k$
appear naturally in classical invariant theory. For example, see e.g. De Concini, Eisenbud and Procesi \cite{DEP},
if $W$ is a direct sum of irreducible $GL_d$-modules $W_d(\lambda)$, the vector subspace $W^{UT_d}$ of the
invariants of the subgroup $UT_d=UT_d(K)$ of all upper unitriangular matrices of $GL_d$
is spanned by elements $w_{\lambda}\in W_d(\lambda)$, where up to a multiplicative constant the element
$w_{\lambda}$ in $W_d(\lambda)$ is the only element with the property
\[
\text{diag}(\xi_1,\ldots,\xi_d)(w_{\lambda})=\Xi^{\lambda}w_{\lambda}=\xi_1^{\lambda_1}\cdots\xi_d^{\lambda_d}w_{\lambda}
\]
and $\text{diag}(\xi_1,\ldots,\xi_d)\in GL_d$ is the diagonal matrix with the elements $\xi_1,\ldots,\xi_d$ at the diagonal.
The vector subspace $W^{SL_d}\subset W^{UT_d}$ of the invariants of $SL_d=SL_d(K)$ is spanned by the elements $w_{\lambda}$ such that
$\lambda=(\lambda_1,\ldots,\lambda_1)$ is a partition in $d$ equal parts. The action of the diagonal subgroup of $GL_d$
on the symmetric algebra $K[W_d(1)^{\oplus k}]$ of
$k$ copies of the canonical $d$-dimensional $GL_d$-module $W_d(1)$ induces a ${\mathbb Z}^d$-grading.
The algebra $K[W_d(1)^{\oplus k}]^{UT_d}$ of $UT_d$-invariants is a graded subalgebra of $K[W_d(1)^{\oplus k}]$. Its Hilbert series is
\[
H(K[W_{1}^{\oplus k}]^{UT_d},T_d)=\sum_{\mu}n_{\mu}T_d^{\mu},
\]
where $n_{\mu}$ is the multiplicity of $S_{\mu}(T_d)$ in $\prod 1/(1-t_i)^k$, i.e. is equal to the corresponding multiplicity series:
\[
H(K[W_{1}^{\oplus k}]^{UT_d},T_d)=M\left(\prod_{i=1}^d\frac{1}{(1-t_i)^k};T_d\right).
\]
Similarly, the algebra $K[W_{1}^{\oplus k}]^{SL_d}$ of $SL_d$-invariants is $\mathbb Z$-graded and its Hilbert series is
\[
H(K[W_{1}^{\oplus k}]^{SL_d},t)=\sum_{\mu}n_{(\mu_1,\ldots,\mu_1)}t^{\mu_1d}=M'\left(\prod_{i=1}^d\frac{1}{(1-t_i)^k};0,\ldots,0,t^d\right),
\]
where $M'(f(T_d);V_d)$ is the multiplicity series of the symmetric function $f(T_d)$ with respect to the variables
$v_i=t_1\cdots t_i$, $i=1,\ldots,d$.
\end{remark}

In the general case
there is  an easy formula which translates Young-derived sequences
in the language of multiplicity series.

\begin{proposition}\label{multiplicity series of Young-derived}
{\rm (Drensky and Genov \cite{DG1})}
Let $f(T_d)$ be the Young-derived of the symmetric function $g(T_d)$.
Then
\[
Y(M(g);T_d)=M(f;T_d)=M\left(g(T_d)\prod_{i=1}^d\frac{1}{1-t_i};T_d\right)
\]
\[
=\prod_{i=1}^d\frac{1}{1-t_i}\sum(-t_2)^{\varepsilon_2}\ldots(-t_d)^{\varepsilon_d}
M(g;t_1t_2^{\varepsilon_2},t_2^{1-\varepsilon_2}t_3^{\varepsilon_3}\ldots
t_{d-1}^{1-\varepsilon_{d-1}}t_d^{\varepsilon_d},t_d^{1-\varepsilon_d}),
\]
where the summation runs on all $\varepsilon_2,\ldots,\varepsilon_d=0,1$.
\end{proposition}

\begin{remark}\label{Young rule to p rows}
Applied to $S_{\mu}(T_d)$ when $\mu$ is a partition in $\leq p$ parts
the Young rule gives a sum of Schur functions $S_{\lambda}(T_d)$ for partitions $\lambda$ in $\leq p+1$ parts.
Hence, if the multiplicity series $M(g;T_d)$ of $g(T_d)$ does not depend on $t_{p+1}$, then the multiplicity series
of its Young-derived $Y(M(g);T_d)$ does not depend on $t_{p+2}$.
\end{remark}

The proof of the following lemma can be derived directly from the Young rule, compare with the computing of
the Hilbert series of the free metabelian Lie algebra in \cite{D5}.
Instead, we shall apply Proposition \ref{multiplicity series of Young-derived}.

\begin{lemma}\label{sum S(n-1,1)}
The following equality holds:
\[
\sum_{n\geq 2}S_{(n-1,1)}(T_d)=1+(t_1+\cdots+t_d-1)\prod_{i=1}^d\frac{1}{1-t_i}.
\]
\end{lemma}

\begin{proof}
Starting with $M(S_{(1)};T_d)=t_1$, we calculate $Y(t_1;T_d)$. It is sufficient to work for $d=2$.
\[
M\left(\frac{S_{(1)}}{(1-t_1)(1-t_2)};t_1,t_2\right)
=Y(t_1;T_2)=\frac{t_1-t_2(t_1t_2)}{(1-t_1)(1-t_2)}=\frac{t_1+t_1t_2}{1-t_1}
\]
\[
=\sum_{n\geq 0}t^n-1+\sum_{n\geq 2}t_1^{n-1}t_2
\]
\[
=M\left(\sum_{n\geq 0}S_{(n)}-1+\sum_{n\geq 2}S_{(n-1,1)};t_1,t_2\right).
\]
Since the answer is the same for all $d\geq 2$, this is equivalent to
\[
S_{(1)}(T_d)\prod_{i=1}^d\frac{1}{1-t_i}=\prod_{i=1}^d\frac{1}{1-t_i}-1
+\sum_{n\geq 2}S_{(n-1,1)}(T_d),
\]
and this completes the proof because
\[
S_{(1)}(T_d)=t_1+\cdots+t_d.
\]
\end{proof}

The following proposition expresses the Hilbert series of the product of two T-ideals
in terms of the Hilbert series of the factors, and gives the corresponding relation
for the Hilbert series of the relatively free algebras.

\begin{proposition}\label{Hilbert series of products of ideals}
{\rm (Formanek \cite{F2}, see also Halpin \cite{H})}
Let $R_1, R_2$ and $R$ be PI-algebras such that $T(R)=T(R_1)T(R_2)$. Then
\[
H(K\langle X_d\rangle,T_d)H(K\langle X_d\rangle\cap T(R),T_d)
=H(K\langle X_d\rangle\cap T(R_1),T_d)H(K\langle X_d\rangle\cap T(R_2),T_d),
\]
\[
H(F_d(R),T_d)=H(F_d(R_1),T_d)+H(F_d(R_2),T_d)
\]
\[
+(t_1+\cdots+t_d-1)H(F_d(R_1),T_d)H(F_d(R_2),T_d).
\]
\end{proposition}

\begin{proof}
The first equation is presented in \cite{F2} and \cite{H}. The second equation follows immediately
from the first if we take into account that
\[
H(K\langle X_d\rangle,T_d)=\frac{1}{1-(t_1+\cdots+t_d)}
\]
and for any unital algebra $A$ the isomorphism
\[
F_d(A)\cong K\langle X_d\rangle/T(A)
\]
implies
\[
H(F_d(A),T_d)=H(K\langle X_d\rangle,T_d)-H(T(A),T_d).
\]
\end{proof}

Finally, there is an explicit description of the relatively free algebra
$F_d(U_k)$ and its Hilbert series in terms of tensor products of irreducible $GL_d$-modules
and products of Schur functions.

\begin{proposition}\label{formula of Drensky and Kasparian}
{\rm (Drensky and Kasparian \cite{DK}, see also \cite{D6})}
{\rm (i)} Let $W_d(\lambda)$ be the irreducible $GL_d$-module corresponding to the partition $\lambda$.
Then the following $GL_d$-module isomorphism holds:
\[
F_d(U_k)\cong \left(\bigoplus_{n\geq 0}W_d(n)\right)
\otimes\left(\bigoplus_{r=0}^{k-1}\bigoplus_{p_i\geq 2}W_d(p_1-1,1)\otimes\cdots\oplus W_d(p_r-1,1)\right).
\]
{\rm (ii)} The Hilbert series of $F_d(U_k)$ has the form
\[
H(F_d(U_k),T_d)=\prod_{i=1}^d\frac{1}{1-t_i}\sum_{r=0}^{k-1}\sum_{p_i\geq 2}S_{(p_1-1,1)}(T_d)\cdots S_{(p_r-1,1)}(T_d).
\]
\end{proposition}

\begin{proof}
By a theorem of Drensky \cite{D2} (or \cite[Theorems 4.3.12 and 12.5.4]{D6})
if $A$ is any unital PI-algebra and $B_d(A)\subset F_d(A)$ is the vector space of the so called ``proper'' polynomials
in the relatively free algebra $F_d(A)$, then the Hilbert series of $F_d(A)$ and $B_d(A)$ are related by
\[
H(F_d(U_k),T_d)=H(K[X_d],T_d)H(B_d(U_k),T_d)=\prod_{i=1}^d\frac{1}{1-t_i}H(B_d(U_k),T_d)
\]
and the following $GL_d$-module isomorphism holds:
\[
F_d(A)\cong K[X_d]\otimes B_d(A)\cong \left(\bigoplus_{n\geq 0}W_d(n)\right)\otimes B_d(A).
\]
Now we apply the decomposition in \cite{DK}, see also \cite[Theorem 12.5.6]{D6},
\[
B_d(U_k)\cong \bigoplus_{r=0}^{k-1}\bigoplus_{p_i\geq 2}W_d(p_1-1,1)\otimes\cdots\oplus W_d(p_r-1,1)
\]
and its counterpart in the language of Schur functions
\[
H(B_d(U_k),T_d)=\sum_{r=0}^{k-1}\sum_{p_i\geq 2}S_{(p_1-1,1)}(T_d)\cdots S_{(p_r-1,1)}(T_d)
\]
and complete the proof.
\end{proof}

\begin{remark}\label{basis of relatively free algebra}
Proposition \ref{formula of Drensky and Kasparian}
expresses the fact that the basis of the vector space $F_d(U_k)$ consists of the polynomials in $d$ variables
\[
x_1^{a_1}\cdots x_d^{a_d}[x_{i_{11}},x_{i_{12}},\ldots,x_{i_{1p_1}}]\cdots [x_{i_{r1}},x_{i_{r2}},\ldots,x_{i_{rp_r}}],
\]
where $p_j\geq 2$, $j=1,\ldots,r$, $r=0,1,\ldots,k-1$, and the subscripts in the commutators satisfy
$i_{j1}>i_{j2}\leq\cdots\leq i_{jp_j}$. The commutators are left normed, i.e.,
\[
[v_1,v_2,\ldots,v_{p-1},v_p]=[[\ldots [v_1,v_2],\ldots,v_{p-1}],v_p].
\]
For fixed $j$ the commutators $[x_{i_{j1}},x_{i_{j2}},\ldots,x_{i_{jp_j}}]$, $p_j\geq 2$, form also a basis of the commutator ideal
$L_d'/L_d''$ of the free metabelian Lie algebra $L_d/L_d''$, where $L_d$ is the free Lie algebra of rank $d$.
It is well known that the $GL_d$-module $L_d'/L_d''$ is a direct sum of $W_d(p-1,1)$ participating with multiplicity 1, $p\geq 2$,
and this means that the Hilbert series of $L_d'/L_d''$ is
\[
H(L_d'/L_d'',T_d)=\sum_{p\geq 2}S_{(p-1,1)}(T_d).
\]
\end{remark}

\begin{remark}\label{algebra of constants}
The commutators $[x_{i_1},x_{i_2},\ldots,x_{i_p}]$, $i_1>i_2\leq\cdots\leq i_p$, from
Remark \ref{basis of relatively free algebra}
generate the subalgebra of $K\langle X_d\rangle$ of the so called ``proper'' polynomials.
By Gerritzen \cite{Ge} it is a free algebra and these commutators form a free generating set
(implicitly this is also in \cite{DK}).
This algebra is also the algebra of constants (i.e., the intersection of the kernels) of the formal partial derivatives
$\partial/\partial x_i$, $i=1,\ldots,d$, as described by
Falk \cite{Fa}. (The freedom does not
follow immediately  from invariant theory of free algebras as developed by
Lane \cite{La} and Kharchenko \cite{Kh}. The derivations
$\partial/\partial x_i$ are locally nilpotent and the exponential automorphisms
$\exp(\partial/\partial x_i)$ are affine automorphisms of
$K\langle X_d\rangle$.
But we cannot apply directly \cite{La} and \cite{Kh}
because these automorphisms are not linear.)
Specht \cite{S} applied proper polynomials in the study of PI-algebras,
see the book \cite{D6} for further applications. To the best of our knowledge
the papers by Specht \cite{S} and Malcev \cite{M} in 1950
are the first publications  where representation theory of symmetric groups was involved in the study of PI-algebras.
Later Regev in a series of brilliant papers, starting with the theorems for the exponential growth of the codimension sequence
and the tensor product of PI-algebras \cite{R1} in 1972, developed his powerful method for quantitative study of PI-algebras.
The present paper is one of the many others which exploit the ideas of Regev.
\end{remark}

\section{Main Results}

We start the exposition of the main results of our paper with the Hilbert series of the relatively
free algebras of the variety generated by the algebra of $k\times k$ upper triangular matrices.

\begin{theorem}\label{Hilbert series of relatively free algebra}
The Hilbert series $H(F_d(U_k),T_d)$ of the algebra $F_d(U_k)$ is
\[
H(F_d(U_k),T_d)=\frac{1}{t_1+\cdots +t_d-1}\left(\left(1+(t_1+\cdots+t_d-1)\prod_{i=1}^d\frac{1}{1-t_i}\right)^k-1\right)
\]
\[
=\sum_{j=1}^k\binom{k}{j}\left(\prod_{i=1}^d\frac{1}{1-t_i}\right)^j(t_1+\cdots+t_d-1)^{j-1}.
\]
\end{theorem}

\begin{proof}
One possible way to prove the theorem is the following. Since
$T(U_k)=C^k$, where $C$ is the commutator ideal of $K\langle X\rangle$,
we have that $T(U_k)=T(U_1)T(U_{k-1})$. Proposition \ref{Hilbert series of products of ideals}
gives the recurrent formula
\[
H(F_d(U_k),T_d)=H(F_d(U_1),T_d)+H(F_d(U_{k-1}),T_d)
\]
\[
+(t_1+\cdots+t_d-1)H(F_d(U_1),T_d)H(F_d(U_{k-1}),T_d).
\]
Since $F_d(U_1)=K\langle X_d\rangle/C\cong K[X_d]$, we start with the well known
\[
H(F_d(U_1),T_d)=H(K[X_d],T_d)=\prod_{i=1}^d\frac{1}{1-t_i}
\]
and complete the proof by induction. Instead, we provide a direct proof. By
Proposition \ref{formula of Drensky and Kasparian} (ii)
\[
H(F_d(U_k),T_d)=\prod_{i=1}^d\frac{1}{1-t_i}\sum_{r=0}^{k-1}\left(\sum_{p_1\geq 2}S_{(p_1-1,1)}(T_d)\right)\cdots
\left(\sum_{p_r\geq 2}S_{(p_r-1,1)}(T_d)\right).
\]
Applying Lemma \ref{sum S(n-1,1)} we obtain
\[
H(F_d(U_k),T_d)=\left(\prod_{i=1}^d\frac{1}{1-t_i}\right)\sum_{r=0}^{k-1}
\left(1+(t_1+\cdots+t_d-1)\prod_{i=1}^d\frac{1}{1-t_i}\right)^r
\]
\[
=\left(\prod_{i=1}^d\frac{1}{1-t_i}\right)\frac{\left(1+(t_1+\cdots+t_d-1)\prod_{i=1}^d\frac{1}{1-t_i}\right)^k-1}
{\left(1+(t_1+\cdots+t_d-1)\prod_{i=1}^d\frac{1}{1-t_i}\right)-1}
\]
\[
=\left(\prod_{i=1}^d\frac{1}{1-t_i}\right)\frac{\left(1+(t_1+\cdots+t_d-1)\prod_{i=1}^d\frac{1}{1-t_i}\right)^k-1}
{(t_1+\cdots+t_d-1)\prod_{i=1}^d\frac{1}{1-t_i}}
\]
\[
=\frac{\left(1+(t_1+\cdots+t_d-1)\prod_{i=1}^d\frac{1}{1-t_i}\right)^k-1}{t_1+\cdots+t_d-1}
\]
\[
=\sum_{j=1}^k\binom{k}{j}\left(\prod_{i=1}^d\frac{1}{1-t_i}\right)^j(t_1+\cdots+t_d-1)^{j-1}.
\]
\end{proof}

Theorem \ref{Hilbert series of relatively free algebra} can be restated in the following way
which, combined with Proposition \ref{multiplicity series of Young-derived},
gives an algorithm to compute the multiplicity series of $U_k$.

\begin{corollary}\label{algorithm for multiplicities}
Let $Y$ be the linear operator in ${\mathbb C}[[V_d]]\subset {\mathbb C}[[T_d]]$ which sends the multiplicity series of
the symmetric function $g(T_d)$ to the multiplicity series of its Young-derived:
\[
Y(M(g);T_d)=M\left(g(T_d)\prod_{i=1}^d\frac{1}{1-t_i};T_d\right).
\]
Then the multiplicity series of $U_k$ is
\[
M(U_k;T_d)=\sum_{j=1}^k\sum_{q=0}^{j-1}\sum_{\lambda\vdash q}(-1)^{j-q-1}\binom{k}{j}\binom{j-1}{q}d_{\lambda}Y^j(T_d^{\lambda}),
\]
where $d_{\lambda}$ is the degree of the irreducible $S_n$-character $\chi_{\lambda}$ and
$T_d^{\lambda}=t_1^{\lambda_1}\cdots t_d^{\lambda_d}$ for $\lambda=(\lambda_1,\ldots,\lambda_d)$.
\end{corollary}

\begin{proof}
Expanding the expression of $H(F_d(U_k),T_d)$ from Theorem \ref{Hilbert series of relatively free algebra} we obtain
\[
H(F_d(U_k),T_d)=\sum_{j=1}^k\binom{k}{j}\left(\prod_{i=1}^d\frac{1}{1-t_i}\right)^j(t_1+\cdots+t_d-1)^{j-1}
\]
\[
=\sum_{j=1}^k\binom{k}{j}\left(\prod_{i=1}^d\frac{1}{1-t_i}\right)^j\sum_{q=0}^{j-1}(-1)^{j-q-1}\binom{j-1}{q}(t_1+\cdots+t_d)^q.
\]
Now we use the well known equality
\[
(t_1+\cdots+t_d)^q=S_{(1)}^q(T_d)=\sum_{\lambda\vdash q}d_{\lambda}S_{\lambda}(T_d).
\]
In the language of representation theory of the symmetric group applied to PI-algebras, it means that the $S_q$-character
of the multilinear component $P_q$ of degree $q$ of the free algebra $K\langle X\rangle$ has the same decomposition as the
character of the regular representation (i.e., of the character of the group algebra $KS_q$ considered as a left $S_q$-module):
\[
\chi_{S_q}(P_q)=\chi_{S_q}(KS_q)=\sum_{\lambda\vdash q}d_{\lambda}\chi_{\lambda}.
\]
Hence
\[
H(F_d(U_k),T_d)=\sum_{j=1}^k\binom{k}{j}\left(\prod_{i=1}^d\frac{1}{1-t_i}\right)^j\sum_{q=0}^{j-1}(-1)^{j-q-1}\binom{j-1}{q}
\sum_{\lambda\vdash q}d_{\lambda}S_{\lambda}(T_d)
\]
\[
=\sum_{j=1}^k\sum_{q=0}^{j-1}\sum_{\lambda\vdash q}
(-1)^{j-q-1}\binom{k}{j}\binom{j-1}{q}d_{\lambda}\left(\prod_{i=1}^d\frac{1}{1-t_i}\right)^jS_{\lambda}(T_d).
\]
This completes the proof because $M(S_{\lambda}(T_d);T_d)=T_d^{\lambda}$ and
\[
M\left(\left(\prod_{i=1}^d\frac{1}{1-t_i}\right)^jS_{\lambda}(T_d);T_d\right)=Y^j(M(S_{\lambda}(T_d);T_d))=Y^j(T_d^{\lambda}).
\]
\end{proof}

The following theorem describes the partitions $\lambda$ with $m_{\lambda}(U_k)\not=0$
and the explicit form of the multiplicities for the partitions of ``maximal'' shape.

\begin{theorem}\label{maximal partitions}
{\rm (i)} If $m_{\lambda}(U_k)\not=0$, then $\lambda=(\lambda_1,\ldots,\lambda_{2k-1})$ is a partition in not more than
$2k-1$ parts and $\overline{\lambda}=(\lambda_{k+1},\ldots,\lambda_{2k-1})$ is a partition of $i\leq k-1$.

{\rm (ii)} If $\overline{\lambda}$ is a partition of $k-1$, then
\[
m_{\lambda}(U_k)=d_{\overline{\lambda}}\dim(W_k(\lambda_1,\ldots,\lambda_k)),
\]
where $d_{\overline{\lambda}}$ is the degree of the $S_{k-1}$-character $\chi_{\overline{\lambda}}$ and
\[
\dim(W_k(\lambda_1,\ldots,\lambda_k))=S_{\lambda}(\underbrace{1,\ldots,1}_{k\text{ \rm times}})
=\prod_{1\leq i<j\leq k}  \frac{\lambda_i - \lambda_j +j -i}{j-i}.
\]
\end{theorem}

\begin{proof}
(i) By Theorem \ref{Hilbert series of relatively free algebra} and in the spirit of
Corollary \ref{algorithm for multiplicities}
the nonzero multiplicities $m_{\lambda}(U_k)$ in the cocharacter sequence of $U_k$ come from the decomposition
as an infinite sum of Schur functions of
\[
\left(\prod_{i=1}^d\frac{1}{1-t_i}\right)^j(t_1+\cdots+t_d)^q
=(S_{(1)}(T_d))^q\left(\prod_{i=1}^d\frac{1}{1-t_i}\right)^j,
\]
$j\leq k$, $q\leq k-1$. The Schur functions $S_{\pi}(T_d)$ participating in the
product $(\prod_{i=1}^d 1/(1-t_i))^j$ are indexed by partitions $\pi$ in $\leq j\leq k$ parts.
By the Branching theorem the multiplication of $S_{\pi}(T_d)$ by $S_{(1)}(T_d)$
gives a sum of $S_{\rho}(T_d)$ where the diagrams $[\rho]$ are obtained from the diagram $[\pi]$ by adding a box.
Clearly $[\rho]$ has not more than one box below the $k$-th row. Multiplying $q$ times by $S_{(1)}(T_d)$
we add to the diagram $[\pi]$ not more than $q\leq k-1$ boxes below the $k$-th row. In this way, if $m_{\lambda}(U_k)\not=0$,
$\lambda=(\lambda_1,\ldots,\lambda_n)$, then
$\lambda_{k+1}+\cdots+\lambda_n\leq k-1$ and obviously $\lambda_{2k}=0$.

(ii) It follows from the proof of (i) that the multiplicity $m_{\lambda}(U_k)$ for $\overline{\lambda}\vdash k-1$
comes from the products of $S_{\mu}(T_d)$ and $(S_{(1)}(T_d))^{k-1}$, where $S_{\mu}(T_d)$ participates in the decomposition
of $(\prod_{i=1}^d 1/(1-t_i))^k$, the diagram $[\mu]=[\mu_1,\ldots,\mu_k]$ has exactly $k$ rows,
and all $k-1$ boxes added to $[\mu]$ to obtain $[\lambda]$ when multiplying $k-1$ times by $S_{(1)}(T_d)$ form the rows of $[\lambda]$
below the first $k$ rows. Hence $\lambda_i=\mu_i$, $i=1,\ldots,k$.
We fill in with $i$ the box of the diagram corresponding to the $i$-th
factor $S_{(1)}(T_d)$. As in the proof of Proposition \ref{result about multiplicities}, the $k-1$ boxes below the
$k$-th row of the diagram $[\lambda]$ are filled in with the integers $1,\ldots,k-1$ and form a standard $\overline{\lambda}$-tableau.
Again, there is a bijection between the standard $\overline{\lambda}$-tableaux
and the summands $S_{\lambda}(T_d)$ obtained from a given $S_{\mu}(T_d)=S_{(\lambda_1,\ldots,\lambda_k)}(T_d)$.

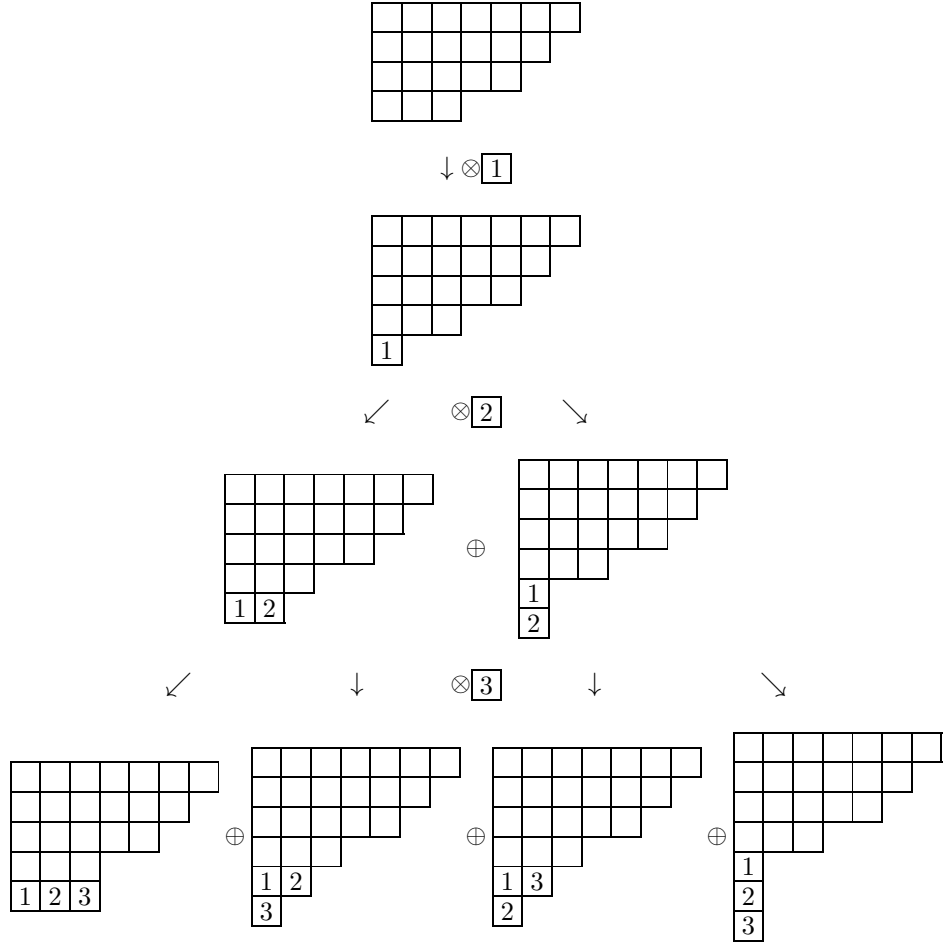
\bigskip

\begin{center}
{\normalsize \Yvcentermath1
\[
\begin{matrix}
\young(\:\:\:\:\:\:\:,\:\:\:\:\:\:,\:\:\:\:\:,\:\:\:)\\
\\
\downarrow\otimes \young(1)\\
\\
\young(\:\:\:\:\:\:\:,\:\:\:\:\:\:,\:\:\:\:\:,\:\:\:,1)\\
\\
\swarrow\quad\quad\otimes\young(2)\quad\quad \searrow\\
\\
\young(\:\:\:\:\:\:\:,\:\:\:\:\:\:,\:\:\:\:\:,\:\:\:,12)\quad\oplus\quad \young(\:\:\:\:\:\:\:,\:\:\:\:\:\:,\:\:\:\:\:,\:\:\:,1,2)\\
\\
\swarrow\quad\quad\quad\quad\quad\quad\downarrow\quad\quad\quad\otimes\young(3)\quad\quad\quad\downarrow\quad\quad\quad\quad\quad\quad \searrow\\
\\
\young(\:\:\:\:\:\:\:,\:\:\:\:\:\:,\:\:\:\:\:,\:\:\:,123)\oplus \young(\:\:\:\:\:\:\:,\:\:\:\:\:\:,\:\:\:\:\:,\:\:\:,12,3)
\oplus \young(\:\:\:\:\:\:\:,\:\:\:\:\:\:,\:\:\:\:\:,\:\:\:,13,2)\oplus \young(\:\:\:\:\:\:\:,\:\:\:\:\:\:,\:\:\:\:\:,\:\:\:,1,2,3)\\
\\
k=4,\quad (\lambda_1,\lambda_2,\lambda_3,\lambda_4)=(7,6,5,3)\\
\end{matrix}
\]
}
\end{center}

\bigskip

\noindent
By Proposition \ref{result about multiplicities}
the multiplicity of $S_{\mu}(T_d)$ in the product $(\prod_{i=1}^d 1/(1-t_i))^k$ is equal to $\dim(W_k(\mu))$.
This completes the proof because the number of the standard $\overline{\lambda}$-tableaux is equal to the degree
$d_{\overline{\lambda}}$ of the irreducible $S_{k-1}$-character $\chi_{\overline{\lambda}}$.
\end{proof}

Using standard procedures of Maple only we have written a small program which computes the multiplicity series of $U_k$.
It is more convenient to state the results of the computations for the difference
$M'(U_k;V_d)-M'(U_{k-1};V_d)$ than for $M'(U_k;V_d)$ or $M(U_k;T_d)$.
(Recall that $M'(R;V_d)=M(R;T_d)$, where $v_i=t_1\cdots t_i$, $i=1,\ldots,d$.)
We give the results for the first several $k$. We write $M'(U_k;V)$ instead of $M'(U_k;V_d)$
assuming that the number of variables $d$ is sufficiently large and has the property
that $m_{\lambda}(U_k)=0$ if $\lambda_{d+1}\not=0$. The case $k=1$ (when $U_1=K$) is trivial and
the case $k=2$ is known (and can be computed applying the formula of Berele and Regev given in the introduction,
see \cite{MRZ}, or directly by the Young rule and Proposition \ref{formula of Drensky and Kasparian} (i)).
We state the results for completeness.

\begin{theorem}\label{multiplicities for small sizes}
The multiplicity series and the multiplicities of the cocharacter sequence
of the algebra $U_k$ of the $k\times k$ upper triangular matrices for $k=1,2$ are
\[
M'(U_1;V)=\frac{1}{1-v_1},\quad m_{\lambda}(U_1)
=\begin{cases}
1,& \lambda=(\lambda_1)\\
0,&\lambda_2>0;\\
\end{cases}
\]
\[
M'(U_2;V)-M'(U_1;V)=\frac{v_2+v_3}{(1-v_1)^2(1-v_2)},
\]
\[
m_{\lambda}(U_2)-m_{\lambda}(U_1)
=\begin{cases}
\lambda_1-\lambda_2+1,&\lambda=(\lambda_1,\lambda_2),\lambda_2>0,\\
 \lambda_1-\lambda_2+1,&\lambda=(\lambda_1,\lambda_2,1),\\
 0,& \text{for all other }\lambda.\\
\end{cases}
\]
\end{theorem}

The results for $U_3$ are the following.

\begin{theorem}\label{3 x 3 matrices}
{\rm (i)} The difference of the multiplicity series of $U_3$ and $U_2$ is
\[
M'(U_3;V)-M'(U_2;V)=\left(\frac{v_5+v_4^2+4v_4+4v_3}{1-v_3}+v_2^2\right)\frac{1-v_1v_2}{(1-v_1)^3(1-v_2)^3}
\]
\[
-\frac{(v_2^2-v_1-3v_2+3)v_4+(v_1v_2^2-v_1v_2+v_2^2-v_1-4v_2+4)v_3}{(1-v_1)^3(1-v_2)^3};
\]

{\rm (ii)} The explicit form of the corresponding multiplicities is
\[
m_{\lambda}(U_3)-m_{\lambda}(U_2)
=\begin{cases}
n_{\lambda},&\lambda=(\lambda_1,\lambda_2,\lambda_3,2),\\
n_{\lambda},&\lambda=(\lambda_1,\lambda_2,\lambda_3,1,1),\\
4n_{\lambda}-c_{\lambda},&\lambda=(\lambda_1,\lambda_2,\lambda_3,1),\\
4n_{\lambda}-c_{\lambda},&\lambda=(\lambda_1,\lambda_2,\lambda_3),\lambda_3>0,\\
\frac{1}{2}\lambda_1(\lambda_1-\lambda_2+1)(\lambda_2-1),&\lambda_2\geq 2,\\
 0,& \text{for all other }\lambda,\\
\end{cases}
\]
where
\[
n_{\lambda}=\dim(W_3(\lambda_1,\lambda_2,\lambda_3))
=\frac{1}{2}(\lambda_1-\lambda_2+1)(\lambda_2-\lambda_3+1)(\lambda_1-\lambda_3+2)
\]
and the correction $c_{\lambda}$ is
\[
c_{\lambda}
=\begin{cases}
\frac{1}{2}(\lambda_1+2)(\lambda_1-\lambda_2+1)(\lambda_2+1),&\lambda=(\lambda_1,\lambda_2,1,1),\\
\frac{1}{2}(\lambda_1+3)(\lambda_1-\lambda_2+1)(\lambda_2+2),&\lambda=(\lambda_1,\lambda_2,1),\\
0,& \text{for all other }\lambda .\\
\end{cases}
\]
\end{theorem}

\begin{proof}
(i) We have evaluated the multiplicity series of $U_3$ applying the algorithm from Corollary \ref{algorithm for multiplicities}.
Instead, we may apply Remark \ref{how to check the proof}: We want to show that the rational function
given in the statement (i) of the theorem and depending
on the five variables $v_1,\ldots,v_5$ is equal to $M'(U_3;V)-M'(U_2;V)$. It is sufficient to check that
the only symmetric function with this rational function as its multiplicity series is equal to
$H(F_d(U_3),T_d)-H(F_d(U_2),T_d)$.
This can be easily verified using the formula from Remark \ref{how to check the proof}.

(ii) We have expanded into a power series the rational expression for $M'(U_3;V)-M'(U_2;V)$ given in part (i) of the theorem
using the equalities
\[
\frac{v_1^{a_1}v_2^{a_2}}{(1-v_1)^3(1-v_2)^3}
=\sum_{n_1\geq a_1}\sum_{n_2\geq a_2}\binom{n_1-a_1+2}{2}\binom{n_2-a_2+2}{2}v_1^{n_1}v_2^{n_2},
\]
\[
\frac{v_1^{a_1}v_2^{a_2}v_3^{a_3}}{(1-v_1)^3(1-v_2)^3(1-v_3)}
=\sum_{n_i\geq a_i}\binom{n_1-a_1+2}{2}\binom{n_2-a_2+2}{2}v_1^{n_1}v_2^{n_2}v_3^{n_3}.
\]
Easy manipulations give the explicit expressions for $m_{\lambda}(U_3)-m_{\lambda}(U_2)$.
In particular,
\[
\frac{1-v_1v_2}{(1-v_1)^3(1-v_2)^3}=\frac{1}{(1-v_1)^3(1-v_2)^2}+\frac{1}{(1-v_1)^2(1-v_2)^3}-\frac{1}{(1-v_1)^2(1-v_2)^2}
\]
\[
=\sum_{n_i\geq 0}\left(\binom{n_1+2}{2}\binom{n_2+1}{1}
+\binom{n_1+1}{1}\binom{n_2+2}{2}-\binom{n_1+1}{1}\binom{n_2+1}{1}\right)v_1^{n_1}v_2^{n_2}
\]
\[
=\frac{1}{2}\sum_{n_i\geq 0}(n_1+1)(n_2+1)(n_1+n_2+2)v_1^{n_1}v_2^{n_2}.
\]
Clearly, the formulas for $\lambda=(\lambda_1,\lambda_2,\lambda_3,2)$ and $\lambda=(\lambda_1,\lambda_2,\lambda_3,1,1)$
follow also from Theorem \ref{maximal partitions}.
\end{proof}

We have computed the multiplicity series $M(U_4;T)$ but the results are too technical to be stated here.

The colength sequence of a PI-algebra $R$ is defined as
the sequence of the number of irreducible characters, counting the multiplicities, in the cocharacter sequence of $R$:
\[
cl_n(R)=\sum_{\lambda\vdash n}m_{\lambda}(R),\quad n=0,1,2,\ldots.
\]
If the algebra $R$ is finite dimensional then the generating function of the colength sequence,
the colength series of $R$,
can be obtained immediately from the multiplicity series $M(R;T_d)$ for a sufficiently large $d$:
\[
cl(R;t)=\sum_{n\geq 0}cl_n(R)t^n=M(R;\underbrace{t,\ldots,t}_{d\text{ \rm times}}).
\]
Theorems \ref{multiplicities for small sizes} and \ref{3 x 3 matrices}
for the multiplicity series of $U_k$ (together with the calculations for $U_4$) give:

\begin{corollary}\label{colength series}
\[
cl(U_1;t)=\frac{1}{1-t};
\]
\[
cl(U_2;t)-cl(U_1;t)=\frac{t^2}{(1-t)^3};
\]
\[
cl(U_3;t)-cl(U_2;t)=\frac{t^4(3+6t+4t^2-2t^3-t^4)}{(1-t)^3(1-t^2)^3};
\]
\[
cl(U_4;t)-cl(U_3;t)=\frac{t^6p(t)}{(1-t)^4(1-t^2)^6},
\]
\[
p(t)=11+45t+63t^2-t^3-42t^4-24t^5+16t^6+12t^7-3t^8-t^9.
\]
\end{corollary}

\end{document}